\documentclass[11pt]{article}
\usepackage[top=2.5cm,bottom=2.5cm,left=2.5cm,right=2.5cm]{geometry}
\usepackage{amssymb}
\usepackage{amsmath,amsthm}
\usepackage[latin1]{inputenc}
\usepackage{color}
\usepackage{tikz,ifthen}
\usepackage{enumerate}
\usepackage{graphicx}
\usepackage{hyperref}
\hypersetup{colorlinks=true, linkcolor=blue, citecolor=blue, urlcolor=blue}

\providecommand{\ZZ}{\mathbb{Z}}
\providecommand{\mdim}{\mathop{\rm mdim}\nolimits}
\providecommand{\edim}{\mathop{\rm edim}\nolimits}

\newtheorem{remark}{Remark}[section]

\newtheorem{lemma}[remark]{Lemma}
\newtheorem{theorem}[remark]{Theorem}
\newtheorem{proposition}[remark]{Proposition}
\newtheorem{corollary}[remark]{Corollary}

\newtheorem{problem}[remark]{Problem}

\begin{document}
\title{Mixed metric dimension of graphs}

\author{Aleksander Kelenc$^{(1)}$, Dorota Kuziak$^{(2)}$, Andrej Taranenko$^{(1,3)}$, Ismael G. Yero$^{(4)}$\\ \\
$^{(1)}${\small Faculty of Natural Sciences and Mathematics}\\
{\small University of Maribor,}  {\small Koro\v{s}ka 160, SI-2000 Maribor, Slovenia.} \\
{\small aleksander.kelenc\@@um.si}, {\small andrej.taranenko\@@um.si}\\
$^{(2)}${\small Departament d'Enginyeria Inform\`atica i Matem\`atiques,}\\
{\small Universitat Rovira i Virgili,} {\small Av. Pa\"{\i}sos Catalans 26, 43007 Tarragona, Spain.}\\
{\small dorota.kuziak\@@urv.cat}\\
$^{(3)}${\small Institute of Mathematics, Physics and Mechanics}\\
{\small Jadranska 19,}  {\small SI-1000 Ljubljana, Slovenia.} \\
$^{(4)}${\small Departamento de Matem\'aticas, Escuela Polit\'ecnica Superior de Algeciras}\\
{\small Universidad de C\'adiz,} {\small
Av. Ram\'on Puyol s/n, 11202 Algeciras, Spain.} \\
{\small ismael.gonzalez\@@uca.es}\\
}

\maketitle

\begin{abstract}
Let $G=(V,E)$ be a connected graph. A vertex $w\in V$ distinguishes two elements (vertices or edges) $x,y\in E\cup V$ if $d_G(w,x)\ne d_G(w,y)$. A set $S$ of vertices in a connected graph $G$ is a mixed metric generator for $G$ if every two elements (vertices or edges) of $G$ are distinguished by some vertex of $S$. The smallest cardinality of a mixed metric generator for $G$ is called the mixed metric dimension and is denoted by $\mdim(G)$. In this paper we consider the structure of mixed metric generators and characterize graphs for which the mixed metric dimension equals the trivial lower and upper bounds. We also give results about the mixed metric dimension of some families of graphs and present an upper bound with respect to the girth of a graph. Finally, we prove that the problem of determining the mixed metric dimension of a graph is NP-hard in the general case.
\end{abstract}



\section{Introduction}

Given a simple and connected graph $G=(V,E)$ and two vertices $x,y\in V$, the distance $d_G(x,y)$ (or $d(x,y)$ for short) between $x$ and $y$ is the length of a shortest $x-y$ path.  A vertex $v\in V$ is said to \emph{distinguish} (we also use the terms ``recognize'' or ``determine'' instead of ``distinguish'') two vertices $x$ and $y$, if $d_G(v,x)\ne d_G(v,y)$. A set $S\subset V$ is called a \emph{metric generator} for $G$ if any pair of vertices of $G$ is distinguished by some element of $S$. A metric generator of minimum cardinality is a \emph{metric basis}, and its cardinality the \emph{metric dimension} of $G$, denoted by $\dim(G)$.

The concept of metric dimension was introduced by Slater in \cite{leaves-trees}, where the metric generators were called \emph{locating sets}, according to some connection with the problem of uniquely recognizing the position of intruders in networks. On the other hand, the concept of metric dimension of a graph was independently introduced by Harary and Melter in \cite{harary}, where metric generators were named \emph{resolving sets}. After these two seminal papers, several works concerning applications, as well as some theoretical properties, of this invariant were published. For instance, applications to the navigation of robots in networks are discussed in \cite{landmarks} and applications to chemistry in \cite{chartrand,chartrand1,pharmacy1}. Furthermore, this topic has found some applications to problems of pattern recognition and image processing, some of which involve the use of hierarchical data structures \cite{Melter1984}. Some interesting connections between metric generators in graphs and the Mastermind game or coin weighing have been presented in \cite{Caceres2007}.

On the other hand, with respect to the theoretical studies on this topic, different points of view of metric generators have been described in the literature, which have highly contributed to gain more insight into the mathematical properties of this parameter related with distances in graphs. Several authors have introduced other variations of metric generators like for instance, resolving dominating sets \cite{brigham}, independent resolving sets \cite{chartrand3}, local metric sets \cite{LocalMetric}, strong resolving sets \cite{Oellermann}, simultaneous metric generators \cite{Ramirez-Cruz-1}, $k$-metric generators \cite{Estrada-Moreno-1}, resolving partitions \cite{chartrand2}, strong resolving partitions \cite{GonzalezYero2013}, $k$-antiresolving sets \cite{Trujillo-1}, etc. have been presented and studied.

Moreover, a few other very interesting articles concerning metric dimension of graphs can be found in the literature. However, according to the amount of results on this topic, we prefer to cite only those papers which are important from our point of view. In concordance with it, we refer the reader to the work \cite{Bailey2011a}, where it can be found some historical evolution, nonstandard terminologies and more references on this topic, and the recent work \cite{Estrada-Moreno-2}, where a general approach on metric generators is described. Some other interesting results and a high number of references can be found in the theses \cite{Estrada-Moreno,Kuziak2014a,Ramirez-Cruz}.

In connection with describing other new variants of metric generators in graph, very recently a parameter used to uniquely recognize the edges of the graph has been introduced in \cite{edge-dim}. Roughly speaking, there was used a graph metric to identify each pair of edges by mean of distances to a fixed set of vertices. This was based on the fact that a metric basis $S$ of a connected graph $G$ uniquely identifies all the vertices of $G$ by mean of distance vectors, but not necessarily such metric basis uniquely recognizes all the edges of the graph. In this sense, the following concepts deserved to be considered.

Given a connected graph $G=(V,E)$, a vertex $v\in V$ and an edge $e=uw\in E$, the distance between the vertex $v$ and the edge $e$ is defined as $d_G(e,v)=\min\{d_G(u,v),d_G(w,v)\}$. A vertex $x\in V$ \emph{distinguishes} (\emph{recognizes} or \emph{determines}) two edges $e_1,e_2\in E$ if $d_G(x,e_1)\ne d_G(x,e_2)$. A set $S$ of vertices in a connected graph $G$ is an \emph{edge metric generator} for $G$ if every two edges of $G$ are distinguished by some vertex of $S$. The smallest cardinality of an edge metric generator for $G$ is called the \emph{edge metric dimension} and is denoted by $\edim(G)$. An \emph{edge metric basis} for $G$ is an edge metric generator for $G$ of cardinality $\edim(G)$.

Having defined the concept of edge metric generator, which uniquely determines every edge of the graph, one could think that probably any edge metric generator $S$ is also a standard metric generator, \emph{i.e.} every vertex of the graph is identified by $S$ or vice versa. However, as it proved in \cite{edge-dim}, this is further away from the reality, although there are several graph families in which such facts occur. In \cite{edge-dim}, among other results, some comparison between these two parameters above were discussed. As a consequence of the study, families of graphs $G$, for which $\edim(G)<\dim(G)$ or $\edim(G)=\dim(G)$ or $\dim(G)<\edim(G)$ hold were described.

In the present work we focus in a kind of mixed version of these two parameters described above. That is, given a connected graph $G$, we wish to uniquely identify the elements (edges and vertices) of $G$ by means of vector distances to a fixed set of vertices of $G$.

Since the (edge or mixed) metric dimension is defined only over connected graphs, in order to avoid repetitions, from now on in this article, all the graph which will be considered are connected, even so we do not explicitly mention it. Moreover, we do not consider here any graph with only one vertex (a singleton). That is, from now on, all the studied graphs contain at least two vertices.

In the next section we formally define mixed metric dimension of a graph and present equivalent definition of the problem in the form of  a linear program. Further, we study the structure of mixed metric generators. We present necessary conditions for a vertex to be included in a mixed metric generator. Moreover, we characterize graphs with extreme mixed metric dimensions (2 or number of vertices). In Section \ref{sec::fam} we present results about the mixed metric dimension of several families of graphs.  Section \ref{sec::bound} is used to give an upper bound for the mixed metric dimension of a graph with respect to the girth of the graph. Finally, in Section \ref{sec::complex} we study the complexity of the problem of determining the mixed metric dimension of a graph and show that it is NP-hard in general. We conclude the paper with three open problems.

\section{Definition of the problem}\label{sec::def}
We say that a vertex $v$ of a connected graph $G$ \emph{distinguishes} two elements (vertices or edges) $x,y$ of $G$ if $d_G(x,v)\ne d_G(y,v)$. A set $S$ of vertices of $G$ is a \emph{mixed metric generator} if any two elements (vertices or edges) of $G$ are distinguished by some vertex of $S$. The smallest cardinality of a mixed metric generator for $G$ is called the \emph{mixed metric dimension} and is denoted by $\mdim(G)$. A \emph{mixed metric basis} for $G$ is a mixed metric generator for $G$ of cardinality $\mdim(G)$.

The problem of determining the mixed metric dimension of a given graph can also be restated as the following optimization problem. Let us now present this mathematical programming model which can be used to solve the problem of computing the mixed metric dimension or finding a mixed metric basis for a graph $G$. A similar model for the case of the standard metric dimension was described in \cite{chartrand}.

Let $G$ be a graph of order $n$ and size $m$ with vertex set $V=\{v_1,v_2,\dots,v_n\}$ and edge set $E=\{e_1,e_2,\dots,e_m\}$. We consider the $n\times (n+m)$ dimensional matrix $D=[d_{ij}]$ such that $d_{ij}=d_G(x_i,x_j)$ and $x_i \in V$ and $x_j\in V\cup E$. Now, given the variables $y_i\in \{0,1\}$ with $i\in \{1,2,\dots,n\}$ we define the following function:
$$\mathcal{F}(y_1,y_2,\dots,y_{n})=y_1+y_2+\dots+y_{n}.$$
Clearly, minimizing the function $\mathcal{F}$ subject to the following constraints
$$\sum_{i=1}^{n}|d_{ij}-d_{il}|y_i\ge 1\;\;\mbox{for every $1\le j<l\le n+m$},$$
is equivalent to finding a mixed metric basis of $G$, since the solution for $y_{1}, y_{2},\dots, y_{n}$ represents a set of values for which the function $\mathcal{F}$ achieves the minimum possible, and this is equivalent to say that the set $W = \{v_i\in V\,:\, y_i=1\}$ is a mixed metric basis for $G$. On the other hand, let $W'$ be a mixed metric basis for $G$ and let $(y'_1,y'_2,\dots,y'_{n})$ be a vector such that for any $i\in \{1,2,\dots,n\}$, $y'_i=0$ if $v_i\notin W'$, or $y'_i=1$ if $v_i\in W'$. Thus, it is straightforward to observe that $\mathcal{F}(y'_1,y'_2,\dots,y'_{n})$ gives a minimum subject to the constraints given before.

\section{The Structure of Mixed Metric Generators}\label{sec::struc}
We next continue with several combinatorial properties of mixed metric generators. Firstly, it clearly follows that any mixed metric generator is also a metric generator and an edge metric generator. In this sense, the following relationship immediately follows. For any graph $G$,
\begin{equation}\label{mdim-dim-edim}
\mdim(G)\ge \max\{\dim(G),\edim(G)\}.
\end{equation}

On the other hand, it is not difficult to see that the whole vertex set of any graph $G$ forms a mixed metric generator. Also, any vertex of $G$ and any incident edge with it, have the same distance to the vertex itself. In this sense, a vertex alone cannot form a mixed metric generator in $G$. As a consequence of these situations, the following remark is readily seen to be true.

\begin{remark}\label{trivial-bounds}
For any graph $G$ of order $n$, $2\le \mdim(G)\le n$.
\end{remark}

First, we present some necessary terminology and several useful propositions about the structure of mixed metric generators. The \emph{open neighbourhood} $N(v)$ of a vertex $v$ in a graph $G$ is given by all the vertices which are adjacent to $v$ and the \emph{closed neighbourhood} of $v$ is $N[v]=N(v)\cup \{v\}$. The vertex $v$ is called an \emph{extreme vertex} if $N(v)$ induces a complete graph. Two vertices $u,v$ of $G$ are called \emph{false twins} if they are have the same open neighbourhoods, \emph{i.e.}, $N(u)=N(v)$. Similarly, the vertices $u,v$ are called \emph{true twins} if $N[u]=N[v]$. A vertex $v$ is a true twin or a false twin in $G$, if there exists $u\ne v$ such that $u,v$ are true twins or false twins, respectively.

\begin{proposition}\label{true-twins}
If $u,v$ are true twins in a graph $G$, then $u,v$ belong to every mixed metric generator for $G$.
\end{proposition}

\begin{proof}
Since $u,v$ are adjacent, it clearly follows that the edge $uv$ and the vertex $v$ have the same distance to every vertex of the graph, except $u$. Similarly, the edge $uv$ and the vertex $u$ have the same distance to every vertex of the graph, except $v$. As a consequence, $u,v$ must belong to every mixed metric generator for $G$.
\end{proof}

\begin{proposition}\label{false-twins}
If $u,v$ are false twins in a graph $G$ and $S$ is a mixed metric generator for $G$, then $\{u,v\}\cap S\ne \emptyset$.
\end{proposition}

\begin{proof}
If $u,v$ are false twins, it clearly follows that they have the same distance to every vertex of $G$ except themselves. Thus, if $S$ is a mixed metric generator for $G$, then at least one of them must belong to $S$.
\end{proof}

\begin{proposition}\label{extremes}
If $u$ is an extreme vertex in a graph $G$, then $u$ belongs to every mixed metric generator for $G$.
\end{proposition}

\begin{proof}
Since $N(u)$ induces a complete graph, for any vertex $v\in N(u)$ it follows that the edge $uv$ and the vertex $v$ have the same distance to every vertex of the graph, except $u$. Therefore, the vertex $u$ must belong to every mixed metric generator for $G$.
\end{proof}

As a direct consequence of Proposition \ref{extremes} we get the following result.

\begin{corollary}\label{leaves}
If $u$ is a vertex of degree 1 in a graph $G$, then $u$ belongs to every mixed metric generator for $G$.
\end{corollary}

We next deal with characterizing the families of graphs achieving the equality in the bounds from Remark \ref{trivial-bounds}.

\begin{theorem}\label{path-complete}
Let $G$ be any graph of order $n$. Then $\mdim(G)=2$ if and only if $G$ is a path.
\end{theorem}

\begin{proof}
By Corollary \ref{leaves} both end-vertices of the path must be in every mixed metric generator, therefore  $\mdim(P_n) \geq 2$.  It is straightforward to observe that for any path $P_n$ the two leaves of the path distinguish all pairs of elements (vertices and/or edges) of the path. It follows that $\mdim(P_n)=2$.

For the converse, assume $G$ satisfies that $\mdim(G)=2$ and let $S=\{u,v\}$ be any mixed metric basis. If there is a neighbour $v'$ of $v$ such that $d(v',u)\ge d(v,u)$, then $d(v'v,u)=d(v,u)$, which means that the edge $v'v$ and the vertex $v$ are not distinguished by any vertex of $S$, a contradiction. Thus, for any vertex $v'$ adjacent to $v$ it follows that $d(v',u)=d(v,u)-1$.

Now, if there exist two vertices $x$ and $y$ belonging to two different shortest $u-v$ paths such that $d(x,u)=d(y,u)$, then also $d(x,v)=d(y,v)$, which means $x,y$ are not distinguished by $S$, a contradiction again.

So, there exists exactly one shortest $u-v$ path in $G$, say $P=uw_1w_2\ldots w_rv$. Suppose there exists $i \in \{1, \ldots, r\}$ such that the vertex $w_i$ in $P$ is of degree at least three and let $w'$ be a neighbour of $w_i$ which is not in $P$. Since $S$ is a mixed metric basis, the edge $w_iw'$ and the vertex $w_i$ are distinguished by some $x\in S$. This means that $d(w_i, x) \not = d(w_iw', x) = \min\{d(w_i, x), d(w',x)\}$. It follows that $d(w',x) < d(w_i,x)$. Let $x' \in S\setminus \{x\}$. Since $d(w',x) \leq d(w_i,x) - 1$, there is a path $Q=x\ldots w'w_i\ldots x'$ of length $d(x, w')+d(w',w_i)+d(w_i,x') \leq d(w_i,x) - 1 + 1 + d(w_i,x') = d(w_i,x)+d(w_i,x')$ from $x$ to $x'$ (note that $\{x,x'\}=\{u,v\}$), a contradiction since this is either a $u-v$ path shorter than $P$ (which is the shortest $u-v$ path) or a path of the same length than $P$ (contradicting the uniqueness of $P$). Thus, every vertex $w_i$, with $i \in \{1, \ldots, r\}$, in $P$ has degree two.

It remains to prove that $u$ and $v$ are both of degree 1. Suppose $u$ is of degree at least 2. Let $u'$ be the neighbour of $u$ which not in $P$. Since $S$ is a mixed metric basis, the vertex $v$ must distinguish the edge $uu'$ and the vertex $u$. It follows that $d(u',v) < d(u,v)$. Following the same line of thought as for the case above we obtain contradiction for all possibilities. Therefore $u$ is of degree 1. Analogously, $v$ is of degree 1. Since $G$ is connected, it follows that $G$ must be a path.
\end{proof}

\begin{lemma}\label{lemma:cardinality_n-1}
Let $v$ be an arbitrary vertex in a graph $G$ and let $S=V(G) \setminus \{v\}$. If $\;\forall w \in N(v),\; \exists\, x \in S:  d(vw,x) \neq d(w,x)$, then $S$ is a mixed metric generator for the graph $G$.
\end{lemma}

\begin{proof}
If we want to prove that $S$ is a mixed metric generator, we have to show that any two elements (vertices or edges) of the graph $G$ are distinguished by some vertex from the set $S$.
Any subset of $V(G)$ with cardinality $n-1$ is a metric generator and also an edge metric generator. So, we only have to check pairs of elements, where one element is a vertex and the other is an edge.
Let $e \in E(G)$ be an arbitrary edge. The vertex $v$ and the edge $e$ are distinguished by at least one endpoint of the edge $e$.
All vertices different from $v$ are in the set $S$. This means that for an arbitrary vertex $u \in V(G)\setminus \{v\}$ we only have to check the edges that are incident with the vertex $u$. If both endpoints of the edge $e=uw$ are in the set $S$, then $u$ and $e$ are distinguished by the vertex $w$.
It remains to check only the pairs of vertices $w$ and edges $wv$ for all $w \in N(v)$. Since we know that for all such pairs there $\exists x \in S:  d(vw,x) \neq d(w,x)$ it follows that $S$ is a mixed metric generator.
\end{proof}

Let $v$ be a vertex of a graph $G$. A vertex $u \in N(v)$ is said to be a \emph{maximal neighbour} of the vertex $v$ if all neighbours of $v$ (and $v$ itself) are also in the closed neighbourhood of $u$.
Now, we are ready to characterize the family of graphs $G$ satisfying that $\mdim(G)=n$.

\begin{theorem}\label{includingAll}
Let $G$ be a graph of order $n$.
Then $\mdim(G)=n$ if and only if every vertex of the graph $G$ has a maximal neighbour.
\end{theorem}

\begin{proof}
First let $\mdim(G)=n$. We want to prove that $\forall v \in V(G), \exists u \in N(v): N[v] \subseteq N[u]$.
Towards contradiction suppose that there $\exists v \in V(G), \forall u \in N(v): N[v] \not\subseteq N[u]$.
Let $S=V(G) \setminus \{v\}$. We claim that $S$ is a mixed metric generator.

If $S$ is not a mixed metric generator, then due to Lemma \ref{lemma:cardinality_n-1} there $\exists w \in N(v), \forall x \in S:  d(vw,x)=d(w,x)$.
Since $w \in N(v)$ it follows that $N[v] \not\subseteq N[w]$, so there exists $v' \in N(v): wv' \notin E(G)$. It follows that $1=d(vw,v') \neq d(w,v')=2 $, a contradiction. So $S$ is a metric generator and $\mdim(G)<n$, a contradiction.

For the converse assume that $\forall v \in V(G), \exists u \in N(v): N[v] \subseteq N[u]$. Suppose that $\mdim(G)<n$. Therefore there exists a mixed metric generator $S$ with cardinality $n-1$ and $v \in V(G): v \notin S $. Let $u \in N(v)$ be a neighbour of $v$ for which it holds that $N[v] \subseteq N[u]$. Since $S$ is a mixed metric generator, there must exist $x \in S$, such that $d(u,x) \neq d(uv,x)$. Thus, it follows that $d(v,x)<d(u,x)$.
On an arbitrary shortest path between $x$ and $v$ there exists $v' \in N(v)$ such that $d(v,x)=d(v',x)+1$. Since $N[v] \subseteq N[u]$ it follows that $d(v,x) \geq d(u,x)$, a contradiction. Therefore $\mdim(G)=n$.
\end{proof}

\section{Mixed Metric Dimension of Some Families of Graphs}\label{sec::fam}
In this section we determine the mixed metric dimension of cycles, complete bipartite graphs, trees and grid graphs.

\begin{proposition}\label{cycle}
For any positive integer $n\ge 4$, $\mdim(C_n)=3$.
\end{proposition}

\begin{proof}
From Remark \ref{trivial-bounds} and Theorem \ref{path-complete} we know that $\mdim(C_n)\ge 3$. On the other hand, let $V(C_n)=\{v_0,v_1,\dots,v_{n-1}\}$ where $v_iv_{i+1}\in E(C_n)$ for every $i\in \{0,\dots,n-1\}$ and operation $i+1$ is done modulo $n$. Let $S=\{v_0,v_1,v_{\left\lceil\frac{n}{2}\right\rceil}\}$. It is clear that the vertices $v_0,v_1$ distinguish every pair of two distinct vertices or two distinct edges. Now, let $e$ be an edge and let $v_i$ be a vertex. If $d(e,v_0)=d(v_i,v_0)$ and $d(e,v_1)=d(v_i,v_1)$, then it must happen either $e=v_iv_{i+1}$ or $e=v_{i-1}v_{i}$. Thus, it follows either $d(e,v_{\left\lceil\frac{n}{2}\right\rceil})=d(v_{i+1},v_{\left\lceil\frac{n}{2}\right\rceil})<d(v_i,v_{\left\lceil\frac{n}{2}\right\rceil})$ or $d(e,v_{\left\lceil\frac{n}{2}\right\rceil})=d(v_{i-1},v_{\left\lceil\frac{n}{2}\right\rceil})<d(v_i,v_{\left\lceil\frac{n}{2}\right\rceil})$. Therefore, the edge $e$ and the vertex $v_i$ are distinguished by $v_{\left\lceil\frac{n}{2}\right\rceil}$ and, as a consequence, $S$ is a mixed metric generator of cardinality three, which completes the proof.
\end{proof}

\begin{proposition}\label{bipart}
For any positive integers $r,t\ge 2$, \[\mdim(K_{r,t})=\left\{\begin{array}{ll}
                                                                    r+t-1, & \text{if } r=2  \text{ or } t=2, \\
                                                                    r+t-2, & \text{otherwise.}
                                                                  \end{array}
\right.\]
\end{proposition}

\begin{proof}
From \cite{Caceres2005} and \cite{edge-dim} we know that $\dim(K_{r,t})=\edim(K_{r,t})=r+t-2$. So, by using (\ref{mdim-dim-edim}) we have $\mdim(K_{r,t})\ge r+t-2$. Let $U$ and $V$ be the bipartition sets of $K_{r,t}$ with $|U|=r$ and $|V|=t$. We first consider the case $r=2$. Suppose $\mdim(K_{r,t})=r+t-2$ and let $S$ be a mixed metric basis for $K_{2,t}$. Since any metric basis or edge metric basis must contain at least $r-1$ vertices of $U$ and $t-1$ vertices of $V$, we deduce that $|U\cap S|=1$ and $|V\cap S|=t-1$. Let $u\in U\cap S$ and $v\in V-S$. We observe that the vertex $u$ has distance 0 to itself (vertex $u$) and distance 1 to every other vertex in $S$. Moreover, the edge $uv$ has distance 0 to the vertex $u$ and distance 1 to every other vertex in $S$. Thus, the vertex $u$ and the edge $uv$ are not distinguished by $S$, a contradiction. A similar contradiction is obtained if $t=2$. Therefore, $\mdim(K_{r,t})\ge r+t-1$ and the proof is completed by using Theorem \ref{includingAll}, since no vertex of $K_{r,t}$ admits a maximal neighbour.

From now on, assume $r,t\ge 3$. Let $S$ be set of cardinality $r+t-2$ such that it does not contain exactly one vertex from each bipartition set of $K_{r,t}$. Since $S$ is a metric basis and also an edge metric basis, we only need to check that $S$ distinguishes those pairs given by an edge and by a vertex. But, this is straightforward to observe since any edge of $K_{r,t}$ has distance 0 or 1 to every vertex of $S$ and for any vertex there is at least one vertex in $S$ at distance 2, since $r\ge 3$ and $t\ge 3$. Therefore, $S$ is a mixed metric generator of cardinality $r+t-2$ and the result follows.
\end{proof}

\begin{theorem}\label{bound-formula}
For any tree $T$ with $l(T)$ leaves, $\mdim(T)=l(T)$.
\end{theorem}

\begin{proof}
Let $S$ be the set of all leaves of $T$ and let $x,y$ be any two distinct elements of $T$. From \cite{landmarks} and \cite{edge-dim} is known that there are a metric basis and an edge metric basis which are both subsets of leaves in $T$. Thus, if $x,y$ are either two vertices or two edges, then they are distinguished by $S$, which is formed by all leaves of $T$. Now, assume $x=x_1x_2$ is an edge and $y$ is a vertex. Without loss of generality, we consider there is an $x_1-y$ path containing $x_2$ (notice that it could happen $y=x_2$). Now, let $x'$ and $y'$ be two leaves of $T$ such that $x_1,x_2,y$ lie in the $x'-y'$ path (notice that it could be $x'=x_1$ and $y'=y$ or viceversa). Thus, it is easy to see that at least one of the leaves $x'$ or $y'$ distinguishes $x$ and $y$. The case when only one of these two leaves distinguishes $x$ and $y$ is given whether $x_2=y$. Therefore, $S$ is a mixed metric generator and we have that $\mdim(T)\le l(T)$.  On the other hand, since every leaf of $T$ is of degree 1 from Corollary \ref{leaves}, we obtain that $\mdim(T)\ge l(T)$, which completes the proof.
\end{proof}

The Cartesian product of two graphs $G$ and $H$ is the graph $G\Box H$, such that $V(G\Box H)=\{(a,b)\;:\;a\in V(G),\;b\in V(H)\}$ and two vertices $(a,b)$ and $(c,d)$ are adjacent in $G\Box H$ if and only if, either ($a=c$ and $bd\in E(H)$), or ($b=d$ and $ac\in E(G)$). Let $h \in V(H)$. We refer to the set $V(G)\times \{h\}$ as a $G$-layer. Similarly $\{g\} \times V(H)$, $g \in V(G)$ is an $H$-layer. When referring to a specific $G$ or $H$ layer, we denote them by $G^h$ or $^gH$, respectively. Obviously, the subgraph induced by a $G$-layer or by an $H$-layer is isomorphic to $G$ or $H$, respectively. Next we give the value of the mixed metric dimension of the grid graph, which is the Cartesian product of two paths $P_r$ and $P_t$ with $r$ and $t$ vertices, respectively.

\begin{proposition}
Let $G$ be the grid graph $G=P_r\Box P_t$, with $r\geq t\geq 2$. Then $\mdim(G)=3$.
\end{proposition}

\begin{proof}
In order to simplify the procedure, we shall embed $G$ into $\ZZ^2$. That is, each vertex of the grid is represented as an ordered pair of coordinates $(x,y)$. In this sense, $G$ is embedded into $\ZZ^2$ where $(0,0),(r-1,0),(0,t-1), (r-1,t-1)$ are the corner vertices of $G$ (the vertices of degree two). We shall prove that the set $S=\{(0,0), (0,t-1), (r-1,0)\}$ is a mixed metric generator for the grid $G$. Consider any two different elements $x,y$ of $G$.

\noindent Case 1: $x,y$ are vertices. From \cite{landmarks} we know that $S'=\{(0,0), (0,t-1)\}$ is a metric generator for $G$. Thus, $x$ and $y$ are distinguished by $(0,0)$ or by $(0,t-1)$. Notice that also  $S=\{(0,0), (r-1,0)$ is a metric generator for $G$.

\noindent Case 2: $x,y$ are edges. From \cite{edge-dim} we know that $S'=\{(0,0), (0,t-1)\}$ or $S=\{(0,0), (r-1,0)$ are edge metric generators for $G$ and we are done for this case.

\noindent Case 3: $x$ is a vertex and $y$ is an edge, say $x=(i,j)$ and $y=(k,a)(k,b)$ (notice that vertices of any edge have either equal first components or equal second components). Without loss of generality we assume $a<b$ (which means $b=a+1$). Suppose the vertex $x$ and the edge $y$ are not distinguished by $S$. This means the following.
\[i+j=d(x,(0,0))=d(y,(0,0))=k+a,\]
\[i+t-1-j=d(x,(0,t-1))=d(y,(0,t-1))=k+t-1-b=k+t-2-a,\]
\[j+r-1-i=d(x,(r-1,0))=d(y,(r-1,0))=a+r-1-k.\]
Thus, we obtain the following system of equations
\begin{align*}
i+j-k-a&=0\\
i-j-k+a&=-1\\
-i+j+k-a&=0
\end{align*}
which is straightforward to observe to be a not compatible system of linear equation, a contradiction. An analogous procedure gives a similar contradiction in the case $x=(i,j)$ and $y=(a,k)(b,k)$. Thus, at least one of the vertices in $S$ identifies the pair $x,y$. As a consequence, $S$ is a mixed metric generator of cardinality three. Therefore, by using Theorem \ref{path-complete} we complete the proof.
\end{proof}

\section{An Upper Bound for the Mixed Metric Dimension of Graphs}\label{sec::bound}
The \emph{girth} $g(G)$ of $G$ is the order of the smallest cycle in $G$. We can give an upper bound for $\mdim(G)$ in terms of the girth of the graph.

\begin{theorem}\label{girth}
Let $G$ be a graph of order $n$. If $G$ has a cycle, then $\mdim(G)\le n-g(G)+3$.
\end{theorem}

\begin{proof}
Let $C=v_0v_1\ldots v_{r-1}v_0$ be a cycle of order $r=g(G)$ in the graph $G$. We claim that $S = V(G) - V(C) \cup \{v_0, v_1, v_{\left\lceil\frac{r}{2}\right\rceil}\}\}$ is a mixed metric generator.

Let $x, y \in V(G)$ be two arbitrary distinct vertices. If at least one of them, say $x$, is in $S$, then they are clearly distinguished by $x$, since $0=d(x,x)\not = d(x,y) > 0$. If none of them is in $S$, then they are vertices of the cycle $C$ and are by Proposition \ref{cycle} distinguished by at least one of $\{v_0, v_1, v_{\left\lceil\frac{r}{2}\right\rceil}\}\}$. Therefore, $S$ is a metric generator.

Now, let $e,f \in E(G)$ be two distinct edges of $G$. If at least one of them, say $e$, has both end-vertices in $S$, then they are clearly distinguished by at least one end-vertex of $e$. Suppose now, that $e=uv$, with $u\in S$ and $v \in V(G)-S$. If $e$ and $f$ are disjoint or their common end-vertex is $v$, then they are distinguished by $u$.  If $e=uv$ and $f=uv'$ and $v, v' \in V(C)$, then the vertex that distinguishes $v$ and $v'$ also distinguishes $e$ and $f$. The remaining case, where $e$ and $f$ have no end-vertices in $S$ is covered by Proposition \ref{cycle}. It follows that $S$ is an edge metric generator.

To conclude the proof we need to prove that any vertex and any edge are distinguished by at least one vertex of $S$. Towards contradiction suppose that there exist $e \in E(G)$ and $v \in V(G)$ that are not distinguished by any vertex of $S$; in other words $\forall x\in S: d(e,x) = d(v, x)$.
Suppose both end-vertices of $e=xy$ are in $S$ (note that it could happen that $v\in\{x,y\}$). Then $e$ and $v$ are distinguished by the  endpoint of $e$ that is not $v$, a contradiction. Suppose that both end-vertices of $e=xy$ are in $V(G)-S$ (again, it could be that $v\in\{x,y\}$). If $v\in S$, then $e$ and $v$ are distinguished by $v$, a contradiction. The case where $v \not \in S$ is covered by the fact that $C$ is a smallest cycle in $G$ and Proposition \ref{cycle}, again a contradiction.
The remaining case is where $e=xy$, with $x\in S$ and $y \in V(G)-S$. If $v$ is not an end-vertex of $e$ or $v=y$, then $e$ and $v$ are distinguished by $x$, a contradiction. Finally, say $v=x$. If $x \in V(C)$, again, since $C$ is a smallest cycle in $G$ at least one vertex of $\{v_0, v_1, v_{\left\lceil\frac{r}{2}\right\rceil}\}$ distinguishes the edge $e$ and vertex $v$ by Proposition \ref{cycle}, a contradiction. Therefore, $x \notin V(C)$.
Let $v' \in  \{v_0, v_1, v_{\left\lceil\frac{r}{2}\right\rceil}\}$ be a vertex closest to $y$. Then $d(e,v') \leq d(y, v') \leq \frac{r}{4}$. On the other hand, since $v'\in S$ by assumption $d(v,v') = d(e,v') \leq d(y,v') \leq \frac{r}{4}$. Let $P_{v',y}$ be the shortest path in $C$ from $v'$ to $y$. Let $P_{v',v}$ be the shortest path in $G$ from $v'$ to $v$. But then the subgraph of $G$ induced by vertices of $P_{v',y}$ and $P_{v',v}$ admits a cycle of size at most $d(v,v')+d(y,v')+d(y,v) \leq \frac{r}{4}+\frac{r}{4}+1 = \frac{r}{2}+1 < r$  (the case where the two paths $P_{v',y}$ and $P_{v',v}$ have no internal vertices in common; otherwise the cycle in question is even smaller), a contradiction with the fact that $r$ is the girth of the graph $G$. Since we obtained a contradiction in all cases, it follows that any vertex and any edge are distinguished by at least one vertex of $S$.

Combining all of the above it follows that $S$ is a mixed metric generator and the proof is completed.
\end{proof}

Clearly the bound from Theorem \ref{girth} is sharp as the following examples show. For any cycle $C_n$, $\mdim(C_n)=n-g(C_n)+3=3$. For any complete graph $\mdim(K_n)=n-g(K_n)+3=n$. For any complete bipartite graph $K_{2,t}$ we have $\mdim(K_{2,t})=t+2-g(K_{2,t})+3=t+1$. For any graph $G$ such that every vertex has a maximal neighbour the girth is $g(G)=3$, therefore by Theorem \ref{includingAll}, $\mdim(G)=n-g(G)+3$.

\section{The Complexity of the Mixed Metric Dimension Problem}\label{sec::complex}

Due to the close relationship between the mixed metric dimension, edge metric dimension and the standard metric dimension, it is natural to think how computationally difficult the problem of computing the mixed metric dimension of a graph is. The decision problems concerning the metric dimension and the edge metric dimension of a graph are already known as NP-complete problems. The proofs are presented in the book \cite{garey} (a formal proof of it appeared in \cite{landmarks}) and in \cite{edge-dim}, respectively.
Let us take a look if the decision problem for the mixed metric dimension is also NP-complete.
We will use a reduction from the 3-SAT problem, as in the case of the metric dimension proof in \cite{landmarks} and edge metric dimension proof in \cite{edge-dim} with slight improvements to the gadgets in construction. From now on, in this section we show that the problem of finding the mixed metric dimension of an arbitrary connected graph is NP-hard. We first deal with the following decision problem.

\[\begin{tabular}{|l|}
  \hline
  \mbox{MIXED METRIC DIMENSION PROBLEM (MDIM problem for short)}\\
  \mbox{INSTANCE: A connected graph $G$ of order $n\ge 3$ and an integer $2\le r\le n$.}\\
  \mbox{QUESTION: Is $\mdim(G)\le r$?}\\
  \hline
\end{tabular}\]\\
To study the complexity of the problem above we make a reduction from the 3-SAT problem, which is one of the most classical problems known to be NP-complete. For more information on this problem, and NP-completeness reductions in general, we suggest \cite{garey}.

\begin{theorem}\label{theorem:NP-complete}
The MDIM problem is NP-complete.
\end{theorem}

\begin{proof}
First let us show that MDIM is in NP. For a set of vertices $S$ guessed by a non-deterministic algorithm for the problem, one needs to check that this is a mixed metric generator. This can be checked in polynomial time. One has to compute the distances from vertices of $S$ to all elements (edges and vertices) and check that all pairs of these elements have different distance vectors with respect to the set $S$.

We now describe a polynomial transformation of the 3-SAT problem to the MDIM problem. Consider an arbitrary input of the 3-SAT problem, a collection $C=\{c_1, c_2, \ldots, c_m\}$ of clauses over a finite set $U=\{u_1,u_2,\ldots,u_n\}$ of Boolean variables. We shall construct a connected graph $G=(V,E)$ and set a positive integer $r \leq |V|$  such that the graph $G$ has a mixed metric generator of size at most $r$ if and only if $C$ is satisfiable. The construction will be made up of several components augmented by some additional edges for communicating between various components.

For each variable $u_i \in U$ we construct a truth-setting component $X_i=(V_i,E_i)$, with $V_i=\{T_i,F_i,a_i,b_i,c_i, d_i\}$ and $E_i=\{T_ic_i,a_ic_i,a_ib_i,b_id_i,c_id_i,d_iF_i\}$ (see Figure \ref{figure:NP1} for reference).
The vertices $T_i$ and $F_i$ are the \texttt{TRUE} and \texttt{FALSE} ends of the component, respectively. Each component is connected with the rest of the graph only through these two vertices which gives us the following proposition.

\begin{proposition}\label{remark:NP1}
Let $u_i$ be an arbitrary variable in $U$. Any mixed metric generator must contain at least one vertex from the set $\{a_i,b_i\}$.
\end{proposition}

\begin{proof}
Suppose that there exists an edge metric generator $S$ without any of these vertices in it. Since the component $X_i$ is attached to the rest of the graph only through the vertices $T_i$ and $F_i$, due to the symmetry, this implies that the vertex $c_i$ and edge $a_ic_i$ have the same distances to all vertices in the set $S$, a contradiction.
\end{proof}

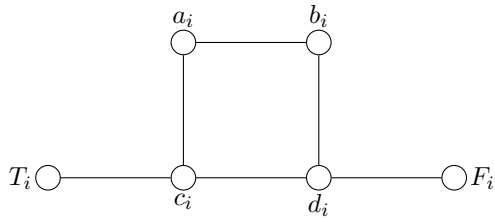
\begin{figure}[!ht]
\centering
\begin{tikzpicture}[scale=0.9, transform shape]

\node [draw, shape=circle] (a1) at (-6,6) {};
\node [draw, shape=circle] (a2) at (-4,6) {};
\node [draw, shape=circle] (a3) at (-2,6) {};
\node [draw, shape=circle] (a4) at (0,6) {};
\node [draw, shape=circle] (a5) at (-2,8) {};
\node [draw, shape=circle] (a6) at (-4,8) {};

\draw (-6.1,6) node[left] {$T_i$};
\draw (0.1,6) node[right] {$F_i$};
\draw (-2,5.9) node[below] {$d_i$};
\draw (-4,5.9) node[below] {$c_i$};
\draw (-2,8.1) node[above] {$b_i$};
\draw (-4,8.1) node[above] {$a_i$};

\foreach \from/\to in {
 a1/a2, a2/a3, a3/a4, a2/a6, a3/a5, a5/a6}
\draw (\from) -- (\to);
\end{tikzpicture}
\caption{The truth-setting component for variable $u_i$.}\label{figure:NP1}
\end{figure}

Now, suppose that $c_j=y_j^1 \vee y_j^2 \vee y_j^3$, where $y_j^k$ is a literal in the clause $c_j$. For such clause $c_j$, we construct a satisfaction testing component $Y_j=(V_j',E_j')$, with $V_j'=\{c_j^1,\ldots, c_j^{6}\}$ and $E_j'=\{c_j^1c_j^2,c_j^2c_j^5,c_j^1c_j^3,c_j^2c_j^4,c_j^6c_j^3,c_j^3c_j^4\}$ (see Figure \ref{figure:NP2} for reference). The component is attached to the rest of the graph only through vertices $c_j^1$ and $c_j^2$ which gives us the following proposition.

\begin{proposition}\label{remark:NP2}
Let $c_j$ be an arbitrary clause in $C$. Any mixed metric generator must contain the vertices $c_j^5$ and $c_j^6$.
\end{proposition}

\begin{proof}
Suppose that there exists an edge metric generator $S$ without vertex $c_j^5$ in it. Since all the shortest paths from any vertex $x\ne c_j^5$ to the vertex $c_j^2$ and to the edge $c_j^2c_j^5$ go through the vertex $c_j^2$, this implies that the vertex $c_j^2$ and the edge $c_j^2c_j^5$ have the same distance to all vertices in the set $S$, a contradiction. A similar argument applies for the vertex $c_j^6$.
\end{proof}

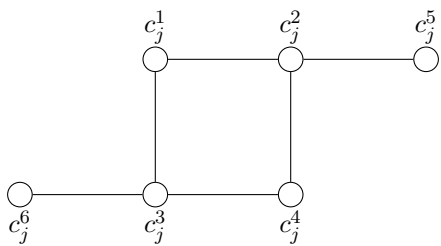
\begin{figure}[!ht]
\centering
\begin{tikzpicture}[scale=.9, transform shape]

\node [draw, shape=circle] (a1) at (-1,0) {};
\node [draw, shape=circle] (a2) at (1,0) {};
\node [draw, shape=circle] (a3) at (-1,-2) {};
\node [draw, shape=circle] (a4) at (1,-2) {};
\node [draw, shape=circle] (a5) at (3,0) {};
\node [draw, shape=circle] (a6) at (-3,-2) {};

\draw (-1,0.1) node[above] {$c_j^1$};
\draw (1,0.1) node[above] {$c_j^2$};
\draw (-1,-2.1) node[below] {$c_j^3$};
\draw (1,-2.1) node[below] {$c_j^4$};
\draw (3,0.1) node[above] {$c_j^5$};
\draw (-3,-2.1) node[below] {$c_j^6$};


\foreach \from/\to in {
 a1/a2, a1/a3, a2/a4, a2/a5, a3/a4, a3/a6}
\draw (\from) -- (\to);
\end{tikzpicture}
\caption{The satisfaction testing component for clause $c_j$.}\label{figure:NP2}
\end{figure}

We also add some edges between truth-setting and satisfaction testing components as follows. If a variable $u_i$ occurs as a positive literal in a clause $c_j$, then we add the edges $T_ic_j^1$ and $F_ic_j^2$. If a variable $u_i$ occurs as a negative literal in a clause $c_j$, then we add the edges $T_ic_j^2$ and $F_ic_j^1$. For each clause $c_j \in C$ denote those six added edges with $E_j''$. We call them \textit{communication} edges. Figure \ref{figure:NP3} shows the edges that were added corresponding to the clause $c_j= (u_1 \vee \overline{u_2} \vee u_3)$, where $\overline{u_2}$ represents the negative literal corresponding to the variable $u_2$.

For all $k \in \{1, \ldots, n\}$ such that neither of $u_k$ and $\overline{u_k}$ occur in clause $c_j$, add the edges $T_kc_j^2$ to the graph $G$. For each clause $c_j \in C$ denote them with $E_j'''$. We call them \textit{neutralizing} edges, because no matter what value is assigned to the variable $u_k$ (or equivalently which vertex $v_k$ from the corresponding truth-setting component $X_k$ is chosen for a mixed metric generator), this gives the same distance from such $v_k$ to the edges $c_j^1c_j^2$ and $c_j^2c_j^4$ from the satisfaction testing component corresponding to the clause $c_j$. These two edges play an important role later in the proof.

Finally, for each clause $c_j$ and every $k \in \{1, \ldots, m\}, k \neq j$, add the edge $c_j^2c_k^2$ to the graph $G$ (if it does not exist). For each clause $c_j \in C$ denote them with $E_j''''$. These edges keep the graph to be connected. We call these edges \textit{correcting} edges.

\begin{figure}[!ht]
\centering
\begin{tikzpicture}[scale=.9, transform shape]

\node [draw, shape=circle] (a1) at (-1,2) {};
\node [draw, shape=circle] (a2) at (1,2) {};
\node [draw, shape=circle] (a3) at (-1,0) {};
\node [draw, shape=circle] (a4) at (1,0) {};
\node [draw, shape=circle] (a5) at (3,2) {};
\node [draw, shape=circle] (a6) at (-3,0) {};

\draw (-1,1.65) node[left] {$c_j^1$};
\draw (1,1.65) node[right] {$c_j^2$};


\node [draw, shape=circle] (a11) at (-6,6) {};
\node [draw, shape=circle] (a12) at (-5,6) {};
\node [draw, shape=circle] (a13) at (-4,6) {};
\node [draw, shape=circle] (a14) at (-3,6) {};
\node [draw, shape=circle] (a15) at (-4,7) {};
\node [draw, shape=circle] (a16) at (-5,7) {};

\draw (-6.1,6) node[left] {$T_1$};
\draw (-2.9,6) node[right] {$F_1$};

\node [draw, shape=circle] (a17) at (3,6) {};
\node [draw, shape=circle] (a18) at (4,6) {};
\node [draw, shape=circle] (a19) at (5,6) {};
\node [draw, shape=circle] (a20) at (6,6) {};
\node [draw, shape=circle] (a21) at (5,7) {};
\node [draw, shape=circle] (a22) at (4,7) {};

\draw (2.9,6) node[left] {$T_3$};
\draw (6.1,6) node[right] {$F_3$};

\node [draw, shape=circle] (a23) at (-1.5,7) {};
\node [draw, shape=circle] (a24) at (-0.5,7) {};
\node [draw, shape=circle] (a25) at (0.5,7) {};
\node [draw, shape=circle] (a26) at (1.5,7) {};
\node [draw, shape=circle] (a27) at (0.5,8) {};
\node [draw, shape=circle] (a28) at (-0.5,8) {};

\draw (-1.6,7) node[left] {$T_2$};
\draw (1.6,7) node[right] {$F_2$};

\foreach \from/\to in {
 a1/a2, a1/a3, a2/a4, a2/a5, a3/a4, a3/a6,
 a11/a12, a12/a13, a13/a14, a13/a15, a15/a16, a16/a12,
 a17/a18, a18/a19, a19/a20, a19/a21, a21/a22, a22/a18,
 a23/a24, a24/a25, a25/a26, a25/a27, a27/a28, a28/a24,
 a11/a1, a14/a2, a17/a1, a20/a2, a23/a2, a26/a1}
\draw (\from) -- (\to);
\end{tikzpicture}
\caption{The subgraph associated to the clause $c_j=(u_1\vee \overline{u_2} \vee u_3)$.}\label{figure:NP3}
\end{figure}
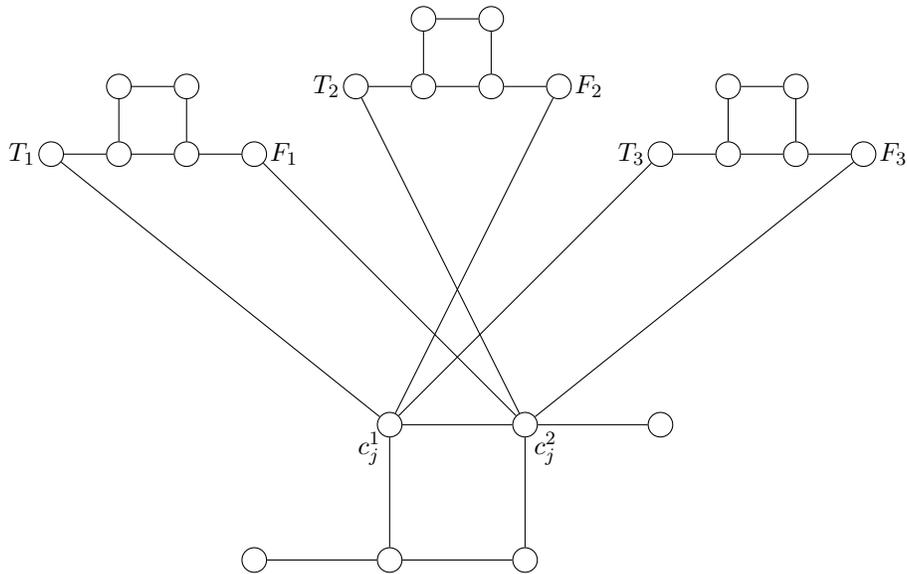

The construction of our instance of the MDIM problem is then completed by setting $r=2m+n$ and $G=(V,E)$, where
\[ V= \left( \bigcup_{i=1}^n V_i \right) \cup \left( \bigcup_{j=1}^m V_j' \right)\]
and
\[ E= \left( \bigcup_{i=1}^n E_i \right) \cup \left( \bigcup_{j=1}^m \left(E_j' \cup E_j'' \cup E_j''' \cup E_j'''' \right)\right).\]

It is not hard too see that the construction can be done in polynomial time. It remains to show that $C$ is satisfiable if and only if $G$ has a mixed metric generator of size $r$. From Propositions \ref{remark:NP1} and \ref{remark:NP2} we get the following.

\begin{corollary}
The mixed metric dimension of the graph $G$ is at least $r=2m+n$.
\end{corollary}

We now continue with the following lemmas which complete the proof of NP-completeness of MDIM problem.

\begin{lemma}\label{lema:NP1}
If $C$ is satisfiable, then the mixed metric dimension of the graph $G$ is $r$.
\end{lemma}

\begin{proof}
We know that the mixed metric dimension is at least $r$. We now construct a mixed metric generator $S$ of size $r$ based on a satisfying truth assignment for $C$.
Let $t:U \rightarrow \{\texttt{TRUE,FALSE}\}$ be a satisfying truth assignment for $C$. For each clause $c_j \in C$ put in the set $S$ vertices $c_j^5$ and $c_j^6$. For each variable $u_i \in U$ put in the set $S$ either the vertex $a_i$ if $t(u_i)=\texttt{TRUE}$, or the vertex $b_i$ if $t(u_i)=\texttt{FALSE}$.
We now show that $S$ is a mixed metric generator for the graph $G$.

Let $e_{j,k}$ be an arbitrary correcting edge between the satisfaction testing components $c_j$ and $c_k$. We notice that $e_{j,k}$ is uniquely determined by the set of vertices $\{c_j^5,c_k^5\}$, because this is the only element in the graph $G$ having distance 1 to both of the vertices $c_j^5$ and $c_k^5$.

Let $i \in \{1, \ldots,n\}$ and $j \in \{1, \ldots,m\}$ be arbitrary indices and let $v_i \in V_i \cap S$. Since we have already checked that any correcting edge is uniquely determined by some vertices in $S$, we do not have to check any pair of elements in which at least one correcting edge occurs.
Also, one can check that each communication edge and each neutralizing edge between a truth-setting component $X_i$ and a satisfaction testing component $Y_j$ is distinguished from all the remaining elements by the vertices $ v_i$, $c_j^5$ and $c_j^6$.

We next take a look at the elements in a truth-setting component. Let $i \in  \{1, \ldots,n\}$ be an arbitrary index and let $x \in V_i \cup E_i$ be an arbitrary element from $X_i$. Since we have already checked that all correcting, communication and neutralizing edges are distinguished from all other elements by some vertices from $S$ we only need to check that $x$ has different distance vectors: (1) from all other elements in $X_i$, (2) from all elements in other truth-setting components, and (3) from all elements in the satisfaction testing components. This is addressed next. (1) For checking that $x$ has different distance vectors to all other elements in $X_i$ suppose that $u_i$ or $\overline{u}_i$ is a literal in clause $c_j$. It is not difficult to check that the vertices $v_i$, $c_j^5$ and $c_j^6$ distinguish the element $x$ from all other elements in $X_i$.
For (2), let $k \in  \{1, \ldots,n\}, k\neq i$, be an arbitrary index. The vertex $v_i$ distinguishes the element $x$ from all elements $x' \in V_k \cup E_k$ (the elements in the truth-setting component $X_k$). For (3), let $j \in  \{1, \ldots,m\}$ be an arbitrary index. Hence, the vertices $c_j^5$ and $c_j^6$ distinguish element $x$ from all elements $y \in V_j' \cup E_j'$ (the elements in the satisfaction testing component $Y_j$).

Finally, we take a look at the elements from the satisfaction testing components.
Let $j \in  \{1, \ldots,m\}$ be an arbitrary index. Every element of $\{c_j^2,c_j^3,c_j^5,c_j^6,c_j^2c_j^5,c_j^3c_j^6\}$ and any other element not covered in previous cases are distinguished by the set of vertices $\{c_j^5,c_j^6\}$. Let $D_1=\{c_j^1c_j^2, c_j^2c_j^4\}$, $D_2=\{c_j^1c_j^3, c_j^3c_j^4\}$ and $D_3=\{c_j^1, c_j^4\}$. The set of vertices $\{c_j^5,c_j^6\}$ also distinguishes any pair of elements where one element is from $D_i$, for $i\in\{1,2,3\}$, and the other element is any element that has not been covered in previous cases and is not in $D_i$.

To complete the proof, we have to show that for any pair $(x,y)$, where $x \not= y$ and $x,y\in D_i$, for some $i\in\{1,2,3\}$  there exists a vertex in the set $S$ that distinguishes $x$ and $y$. Since $C$ is satisfiable, suppose that $c_j$ is satisfied by the variable $u_i$. For the variable $u_i$ there are two possibilities:
\begin{itemize}
\item $u_i$ occurs as a positive literal in $c_j$ and $t(u_i)=\texttt{TRUE}$,
\item $u_i$ occurs as a negative literal in $c_j$ and $t(u_i)=\texttt{FALSE}$.
\end{itemize}
Thus, if $t(u_i)=\texttt{TRUE}$, then we have added the vertex $a_i$ to the set $S$. In such case, the distance from $a_i$ to the edge $c_j^1c_j^2$ is 3, while the distance to the edge $c_j^2c_j^4$ is 4. Similarly, the distance from $a_i$ to the edge $c_j^1c_j^3$ is 3 and to the edge $c_j^3c_j^4$ is 4. The distance from $a_i$ to the vertex $c_j^1$ is 3 and to the vertex $v_j^4$ is 5. The case when $t(u_i)=\texttt{FALSE}$ is symmetric.

Therefore, any two elements of a graph $G$ are distinguished by at least one vertex from the set $S$, and as a consequence, $S$ is a mixed metric generator for a graph $G$, which completes the proof of this lemma.
\end{proof}

\begin{lemma}\label{lema:NP2}
If the mixed metric dimension of graph $G$ is $r$, then $C$ is satisfiable.
\end{lemma}

\begin{proof}
Let $S$ be an arbitrary mixed metric generator for graph $G$ with cardinality $r$. From Propositions \ref{remark:NP1} and \ref{remark:NP2}, the set $S$ must contain at least one vertex from the set $\{a_i,b_i\}$ for each truth-setting component $X_i$ and at least vertices $c_j^5,c_j^6$ from each satisfaction testing component $Y_j$. Since the cardinality of $S$ equals $r=2m+n$, it follows that in the set $S$ there is exactly one vertex from each truth-setting component and exactly two vertices from each satisfaction testing component.
We shall find a function $t: U \rightarrow \{\texttt{TRUE,FALSE}\}$ such that it represents a satisfying truth assignment for the collection of clauses $C$. For an arbitrary $i\in \{1, \ldots, n\}$, let $v_i \in V_i \cap S$. Hence, we define a function $t$ as follows:
\[t(u_i)=\left\{\begin{array}{ll}
                           \texttt{TRUE}, & v_i = a_i, \\
                           \texttt{FALSE}, & v_i = b_i.
                         \end{array}
\right.\]
We shall show that $t$ produces a satisfying truth assignment for $C$. To this end, let $c_j$ be an arbitrary clause. We claim that at least one of its literals has value \texttt{TRUE}. We prove that fact, by tracing which vertex from $S$ distinguishes the edges $e_j^1=c_j^1c_j^2$ and $e_j^2=c_j^2c_j^4$, and showing that the corresponding function $t$ satisfies $c_j$.

Let $k \in  \{1, \ldots,m\}$ be an arbitrary index. For the clause $c_k$ the vertices in the set $S$ are $c_k^5$ and $c_k^6$. If $j=k$, then both edges $e_j^1$ and $e_j^2$ are at distance 1 from $c_k^5$ and at distance 2 from $c_k^6$. If $j \neq k$, then by using the correcting edges, we deduce that the edges $e_j^1$ and $e_j^2$ are at distance 2 from $c_k^5$ and at distance 4 from $c_k^6$. Therefore, none of these vertices distinguish  $e_j^1$ from $e_j^2$.

Now, consider any variable $u_i$ which does not occur in $c_j$. If $v_i = a_i$, then both edges $e_j^1, e_j^2$ are at distance 3 from $v_i$. If $v_i = b_i$, then both edges are at distance 4 from $v_i$. Thus, the vertex of $S$ distinguishing the edges $e_j^1,e_j^2$ must belong to one of the truth-setting components that corresponds to a variable $u_k$ that occurs in the clause $c_j$. We recall that we have added communication edges in such a manner that $v_k$ distinguishes the edges $e_j^1$ and $e_j^2$ only if one of the following statements holds:
\begin{itemize}
\item $u_k$ occurs as a positive literal in $c_j$ and $v_k = a_k$ - in this case $t(u_k)=\texttt{TRUE}$,
\item $u_k$ occurs as a negative literal in $c_j$ and $v_k = b_k$ - in this case $t(u_k)=\texttt{FALSE}$.
\end{itemize}
In both cases the clause $c_j$ is satisfied by the setting assigned to the variable $u_k$. As a consequence, the formula $C$ is satisfiable, which completes the proof of this lemma.
\end{proof}

As a consequence of the Lemmas \ref{lema:NP1} and \ref{lema:NP2} above, the polynomial transformation from 3-SAT to the MDIM problem is done, and the proof of the theorem is now completed.
\end{proof}

As a consequence of Theorem \ref{theorem:NP-complete} we have the following result.

\begin{corollary}\label{np-hard}
The problem of finding the mixed metric dimension of a connected graph is NP-hard.
\end{corollary}

\section{Open problems}
We conclude this paper with three open problems. Considering the close relation between the metric dimension, the edge metric dimension and the mixed metric dimension the following two problems arise naturally. 

\begin{problem}
Characterize graphs $G$ for which $\mdim(G) = \dim(G)$.
\end{problem}

\begin{problem}
Characterize graphs $G$ for which $\mdim(G) = \edim(G)$.
\end{problem}

The bound from Theorem \ref{girth} is achieved for several families of graphs therefore the following problem would also be interesting to explore.

\begin{problem}
Characterize graphs $G$ for which the bound from Theorem \ref{girth} is achieved.
\end{problem}

\end{document}